\documentclass[10pt]{amsart}
\usepackage{amssymb}
\usepackage{graphicx}
\usepackage{xcolor} 
\usepackage{fullpage} 
\usepackage{amsmath}
\usepackage{amsthm}
\usepackage{verbatim}
\usepackage{enumitem}
\setlist[enumerate]{leftmargin=1.5em}
\setlist[itemize]{leftmargin=1.5em}
\definecolor{green}{rgb}{0,0.8,0} 

\newtheorem{maintheorem}{Theorem}

\newtheorem{theorem}{Theorem}[section]

\newtheorem{lemma}[theorem]{Lemma}
\newtheorem{proposition}[theorem]{Proposition}

\theoremstyle{definition}

\theoremstyle{remark}
\newtheorem{remark}[theorem]{Remark}

\numberwithin{equation}{section}
\newcommand{\nrm}[1]{\Vert#1\Vert}

\newcommand{\nnrm}[1]{{\vert\kern-0.25ex\vert\kern-0.25ex\vert #1 
		\vert\kern-0.25ex\vert\kern-0.25ex\vert}}

\newcommand{\lap}{\Delta}

\newcommand{\rd}{\partial}
\newcommand{\nb}{\nabla}

\newcommand{\alp}{\alpha}

\newcommand{\dlt}{\delta}
\newcommand{\Dlt}{\Delta}

\newcommand{\tht}{\theta}

\newcommand{\omg}{\omega}

\newcommand{\bbT}{\mathbb T}

\newcommand{\calI}{\mathcal I}

\newcommand{\calL}{\mathcal L}

\newcommand{\calS}{\mathcal S}

\newcommand{\mrI}{\mathrm{I}}
\newcommand{\mrII}{\mathrm{II}}
\newcommand{\mrIII}{\mathrm{III}}
\newcommand{\mrIV}{\mathrm{IV}}
\newcommand{\mrV}{\mathrm{V}}
\newcommand{\mrVI}{\mathrm{VI}}
\newcommand{\mrVII}{\mathrm{VII}}
\newcommand{\mrVIII}{\mathrm{VIII}}

\newcommand{\tu}{\tilde{u}}					

\newcommand{\lu}{u^{\mathcal{L}}}
\newcommand{\iu}{u^{\mathcal{I}}}
\newcommand{\su}{u^{\mathcal{S}}}
\newcommand{\lw}{\omega^{\mathcal{L}}}
\newcommand{\iw}{\omega^{\mathcal{I}}}
\newcommand{\sw}{\omega^{\mathcal{S}}}
\newcommand{\iA}{A^{\mathcal{I}}}

\newcommand{\normp}[1]{{\left\Vert #1 \right\Vert}_{L^{p}}}
\newcommand{\normif}[1]{{\left\Vert #1 \right\Vert}_{L^{\infty}}}
\newcommand{\normb}[1]{{\left\Vert #1 \right\Vert}_{L^2}}
\newcommand{\normifr}[1]{{\left\Vert #1 \right\Vert}_{L^{\infty}\left (\mathbb{T}^3\right)}}
\newcommand{\normbr}[1]{{\left\Vert #1 \right\Vert}_{L^2\left (\mathbb{T}^3\right)}}
\newcommand{\normhs}[1]{{\left\Vert #1 \right\Vert}_{H^s}}

\newcommand{\normhsps}[1]{{\left\Vert #1 \right\Vert}_{H^{s+1}}}
\newcommand{\normhsr}[1]{{\left\Vert #1 \right\Vert}_{H^{s}(\mathbb{T}^3)}}
\newcommand{\normhssr}[1]{{\left\Vert #1 \right\Vert}_{H^{s-1}(\mathbb{T}^3)}}

\newcommand{\normhspps}[1]{{\left\Vert #1 \right\Vert}_{H^{s+2}}}
\newcommand{\normhspsr}[1]{{\left\Vert #1 \right\Vert}_{H^{s+1}(\mathbb{T}^3)}}
\newcommand{\normhsppsr}[1]{{\left\Vert #1 \right\Vert}_{H^{s+2}(\mathbb{T}^3)}}

\newcommand{\nrmb}[1]{\left\Vert{#1}\right\Vert}

\newcommand{\bu}{\Bar{u}}
\newcommand{\llu}{u^{Lin}}
\newcommand{\ltu}{\Tilde{u}^{Lin}}
\newcommand{\du}{\Tilde{u}^{D}}
\newcommand{\dusu}{\bar{u}^{D*}}

\newcommand{\busu}{\Bar{u}^*}
\newcommand{\piu}{u^{\mathcal{I},P}}
\newcommand{\psu}{u^{\mathcal{S},P}}
\newcommand{\piw}{\omega^{\mathcal{I},P}}
\newcommand{\psw}{\omega^{\mathcal{S},P}}

\begin{document}
	
	\title{Weakened vortex stretching effect in three scale hierarchy for the 3D Euler equations} 

	\author{In-Jee Jeong}
	\address{Department of Mathematics and RIM, Seoul National University.}
	\email{injee\_j@snu.ac.kr}
	
	\author{Jungkyoung Na}
	\address{Department of Mathematics, Brown University.}
	\email{jungkyoung\_na@brown.edu}
	
	\author{Tsuyoshi Yoneda}
	\address{Graduate School of Economics, Hitotsubashi University.}
	\email{t.yoneda@r.hit-u.ac.jp}
	
	\date{\today}
	
	
	
	\maketitle
	
	
	\begin{abstract}
		We consider the 3D incompressible Euler equations under the following three scale hierarchical situation: large-scale vortex stretching the middle-scale, and at the same time, the middle-scale stretching the small-scale. In this situation, we show that, the stretching effect of this middle-scale flow is weakened by the large-scale.  {In other words, the vortices being stretched could have the corresponding stress tensor being weakened.}
	\end{abstract}
	
	
	\section{Introduction}

	Recent direct numerical simulations \cite{Goto-2008,GSK,Motoori-2019,Motoori-2021} of
	the 3D Navier--Stokes turbulence at high Reynolds numbers have shown that
	there exists a \textit{hierarchy} of scale local vortex stretching dynamics.
	In particular, Goto--Saito--Kawahara \cite{GSK} discovered that turbulence at sufficiently high Reynolds numbers in a periodic cube is composed of a self-similar
	hierarchy of anti-parallel pairs of vortex tubes, which is sustained
	by creation of smaller-scale vortices due to stretching in larger-scale strain fields. 
	This discovery has been further investigated by Y.--Goto--Tsuruhashi \cite{YGT} (see also \cite{TGOY}). 
	From these previous results, we could conclude physically  
	that the most important features of the Navier--Stokes turbulence could be scale local vortex stretching which does not seem to be random (see also \cite{JY1,JY2,JY3} for the related mathematical results). Therefore as the sequence of these studies, our next study will be clarifying the locality of vortex stretching dynamics precisely, and in this paper we consider it with three-scale hierarchical structure (c.f. for a geometric approach  considering this locality, see Shimizu--Y. \cite{SY}).
	
	In this paper, we shall consider solutions to the 3D incompressible Euler equations in $\bbT^3$
	\begin{equation}\label{eq:3D-Euler}
		\left\{
		\begin{aligned}
			\rd_t u + u \cdot \nb u + \nb p = 0 ,  \\
			\nb \cdot u = 0,
		\end{aligned}
		\right.
	\end{equation} in which the velocity has the following hierarchical structure: $u^\calL+ u^\calI + u^\calS$, where the letters $\calL,\calI,\calS$ stand for large, intermediate, and small-scale, respectively. They will be arranged in a way that the corresponding vorticities, denoted by $\omg^\calL, \omg^\calI, \omg^\calS$, are mutually almost orthogonal. This is natural since the rate of vortex stretching interaction is maximized when the vortex lines are orthogonal to each other. Indeed, this orthogonality was confirmed in direct numerical simulations \cite{Goto-2008,GSK,Motoori-2019,Motoori-2021} by statistical means. The goal of this paper is to understand the dynamics of vortex stretching under this hierarchical structure. 
	
	Before we describe our results, let us give details of our flow configuration. We take the length scale $L$ of the torus $\mathbb{T}^3:=[-L,L)^3$ to be large ($L=100$ suffices) and fix the large-scale velocity to be the linear strain \begin{equation}\label{eq:u-L}
		\begin{split}
			u^\calL(x) :=  (Mx_1, -Mx_2, 0),
		\end{split}
	\end{equation} in the region $|x|\le 10$, for some large $M\gg 1$. This velocity field corresponds to a large-scale antiparallel columnar vortex parallel to the $x_{3}$-axis supported away from the origin.  {In the following, we shall implicitly assume that $u^\calL$ is a steady solution to 3D Euler with some smooth forcing $f^\calL$ which is supported in the region $|x|>10$.}
	
	To motivate our choice of smaller scale vorticities, consider the \textit{linearized} Euler dynamics around $u^\calL$ in vorticity form: 
	\begin{equation}\label{eq:3D-lin}
		\begin{split}
			\rd_t\omg + u^\calL\cdot\nb\omg = \nb u^\calL \omg 
		\end{split}
	\end{equation} and since \begin{equation*}
		\begin{split}
			\nb u^\calL = \begin{pmatrix}
				M & 0 & 0 \\
				0 & -M & 0 \\
				0 & 0& 0
			\end{pmatrix},
		\end{split}
	\end{equation*} we see that, with respect to the $L^{\infty}$-norm, the solution to \eqref{eq:3D-lin} expands and contracts exponentially in time with rate $Mt$ along the $x_1$ and $x_2$-axis, respectively. 
	
	This raises the following question: if we arrange intermediate and small-scale vorticities to be respectively parallel to $x_1$ and $x_2$ axis, is it possible that the intermediate vortex (being exponentially stretched by the large-scale) significantly stretches the small-scale, by dominating the decay effect of the large-scale? 
	
	It turns out that, interestingly, the answer is no: the small-scale vortex, even in the presence of exponentially stretched intermediate-scale, still decays exponentially in time with the same rate $-Mt$. While this is not too hard to see for the linearized Euler equations, we establish this in the full nonlinear Euler equations: consider the system  \begin{equation}\label{eq:3D-vort}
		\left\{
		\begin{aligned}
			\rd_t \omg + (u + u^{\calL})\cdot \nb \omg = \nb(u + u^\calL) \omg  , \\
			u = \nb\times (-\lap)^{-1}\omg, 
		\end{aligned}
		\right.
	\end{equation} where we shall assume that $\omg = \omg^\calI + \omg^\calS$ is supported in the region $|x|\le 10$, so that \eqref{eq:u-L} applies. We write $u^\calI = \nb\times (-\lap)^{-1}\omg^\calI$ and $u^\calS = \nb\times (-\lap)^{-1}\omg^\calS$. More precisely, we take  \begin{equation}\label{def: iw_0}
		\begin{split}
			\iw_0(x) := (\iw_{1,0}(x_2,x_3),0,0),\qquad \iw_{1,0}(x_2,x_3) := \sum_{i,j \in \{0,1\} } (-1)^{i+j} \varphi( (-1)^i x_2, (-1)^j x_3), 
		\end{split} 
	\end{equation}
	where $\varphi\in C^{\infty}_c$ is a non-negative function supported in $\left\{(x_2,x_3)\in \mathbb{T}^2: 1\le x_2,x_3\le 2 \right\}$. Note that $\iw_0$ is parallel to the $x_1$-axis and has odd symmetry in both $x_2$ and $x_3$. These odd symmetries are imposed to maximize the stretching effects of $\iw$ near the origin and has been inspired by several works \cite{KS,Z,EM}. Finally, with some sufficiently small $0< \varepsilon, \ell \ll 1$, we fix some divergence-free vector field $\tilde{\psi} \in C^{\infty}_{c}(B_{0}(2\ell))$ and require 
	\begin{equation}\label{def: sw_0}
		\sw_0(x) = \varepsilon\tilde{\psi}(x), \qquad \tilde{\psi}(x)=\left(0,\psi\left(\ell^{-1}x\right),0\right), \qquad |x|\le \ell 
	\end{equation} for some smooth bump function $\psi\ge0$. 
	
	\medskip

	We consider the flow map $\Phi$ generated by $\lu+\iu+\su$: namely, $\Phi(0,x)=x$ and \begin{equation}\label{def: flow Phi}
		\frac{d\Phi}{dt}(t,x) = (\lu+\iu+\su)(t,\Phi(t,x)).
	\end{equation} The solution to \eqref{eq:3D-vort} satisfies the Cauchy formula $\omg(t,\Phi)=\nb\Phi \omg_0$, and naturally, we define the evolution of intermediate and small-scale vortex by $\iw (t,\Phi(t,x)) = \nb\Phi(t,x) \iw_{0}(x)$ and $\sw (t,\Phi(t,x)) = \nb\Phi(t,x) \sw_{0}(x)$.  {As long as the solution remains smooth, the support of $\iw$ and $\sw$ remains disjoint.}

	\medskip 
	
	We are now ready to state our main theorem.  {For $0<r<L$, $B_0(r)$ denotes the ball $\{ |x| < r \}$.}
	
	\begin{maintheorem}\label{thm: presence}
		Consider the solution to \eqref{eq:3D-vort} with initial data $\omg_{0} = \iw_0 + \sw_0 \in C^\infty(\bbT^3)$ with $\bbT^3:=[-L,L)^3$. The solution remains smooth in the time interval $[0,T_{M}]$, where $T_M = M^{-1}\log(1+M)$. On this time interval, we have \begin{equation}\label{eq:nb-u-i-decay}
			\begin{split}
				\nrmb{ \nb \iu (t,\cdot) }_{L^\infty(B_{0}(\ell))} \le C\exp(-C^{-1}Mt)  {\nrm{\nb \iu_{0}}_{L^\infty(\bbT^{3})} }
			\end{split}
		\end{equation} and \begin{equation}\label{eq:omega-s-decay}
			\begin{split}
				\nrmb{ \sw (t,\cdot) }_{L^{\infty}\left(  { B_{0}( \tilde{\ell}) }  \right)} \le  C\exp(-Mt)  {\nrm{\sw_{0}}_{L^\infty(\bbT^{3})} }, \qquad \tilde{\ell} = (M^{-1}\ell)^{C} 
			\end{split}
		\end{equation} 
		 {uniformly for all $M\gg1$, $0<\ell\le M^{-C}$, and $0<\varepsilon\le \ell^{2s-3}\exp{(-M^C)}$ with $s>5/2$, where $C>1$ is a constant depending only on $L$, $\psi$ and $\varphi$.}
	\end{maintheorem}

	 \begin{remark}
			While $T_M\to0$ as $M\to\infty$, we have that $MT_M\to\infty$, so that within the timescale of $T_M$, the exponential terms in the right hand sides of \eqref{eq:nb-u-i-decay} and \eqref{eq:omega-s-decay} decays to zero. Furthermore, \eqref{eq:nb-u-i-decay} should be contrasted with exponential growth of $\iw$: $\nrm{\iw(t,\cdot)}_{L^\infty(\bbT^3)} \gtrsim \exp(Mt)\nrm{\iw_0}_{L^\infty(\bbT^3)}$ in the same time interval.  {In other words, the vortices being stretched could have the corresponding stress tensor being weakened.}
	\end{remark}

	\medskip
	
	\noindent \textbf{Interpretation of the result}. It is important to compare the above with the case when the large-scale strain field is absent: in this case, the small-scale vortex gets stretched \textit{at least exponentially in time} by the intermediate-scale (cf. \cite{EM}). Therefore, we see that the vortex stretching effect of the intermediate scale is weakened by the large-scale, and the resulting small-scale vortex dynamics is not really different from the case when the intermediate vortex is absent. Our result suggests that the vortex stretching of ``adjacent'' scales is the one most  
	likely to occur, and thus possible blow-up solution to the three dimensional Navier--Stokes and/or Euler equations may not possess multi-scale vortex stretching motion.
	This insight is consistent with the numerical result by Kang--Yun--Protas \cite{KYP}. Based on solving a suitable optimization problem numerically, they investigated the largest possible growth of the vorticity in finite time in three-dimensional Navier--Stokes flows. Their findings revealed that the flows maximizing the vortex growth exhibit the form of three perpendicular pairs of anti-parallel vortex tubes with the same size (see Figure 11 in \cite{KYP}). Furthermore, the flow evolution resulting from such an initial vorticity is accompanied by reconnection events. We can at least see that this flow does not possess multi-scale vortex stretching motion. 
	
	\begin{remark}
		One could similarly investigate the situation in which the directions of the intermediate and small-scale vorticities are switched. In this case, both large and intermediate-scale vortex stretch the small-scale vortex. However, in this case, the length scale of the intermediate vortex grows at the same time, and escapes the $O(1)$-region around the origin. 
	\end{remark}

	\medskip
	
	\noindent \textbf{Ideas of the proof}.  {Let us briefly explain the main steps of the proof. To begin with, when $\sw$ is completely absent, one can simply study the nonlinear evolution of $\iw$ and obtain the bound \eqref{eq:nb-u-i-decay}. This is already non-trivial since one needs to understand the nonlinear self-interaction of $\iw$. 
		
		Then, one can introduce $\sw$ and formally analyze the vortex stretching equation of $\sw$, which is $D_t(\sw) \simeq (\nb u^\calL + \nb u^\calI) \sw$ if one neglects the nonlinear self-interaction. Applying the bound \eqref{eq:nb-u-i-decay} and integrating, one obtains the desired estimate \eqref{eq:omega-s-decay}. However, in this case, not only one needs to handle the self-interaction term, but also needs to deal with the linear feedback of $\sw$ onto the intermediate scale $\iw$. Indeed, as soon as $\sw$ is introduced,  $\sw$ and  $\iw$ satisfy a coupled system of PDEs, even at the linearized level. In particular, it becomes tricky to obtain the bound \eqref{eq:nb-u-i-decay} in the first place. 
		
		Therefore, our proof consists of a two-step comparison procedure: we introduce the ``pseudo-solution'' pair $(\omg^{\calI,P}, \omg^{\calS,P})$ where $\omg^{\calI,P}$ corresponds to the intermediate scale in the absense of the small scale. Then, we compare $\sw$ with $\omg^{\calS,P}$ which is then compared in turn with the linearized dynamics around $\omg^{\calI,P}$.  
	}

	\subsection*{Notation}
	We employ the letters $C,\,C_1,\,C_2,\cdots$ to denote any constants which may change from line to line in a given computation. In particular, the constants depend only on $L$, $\psi$ and $\varphi$.
	We sometimes use $A\approx B$ and $A\lesssim B$, which mean $A=CB$ and $A\le CB$, respectively, for some constant $C$.

	\section{Preliminaries}
	
	 {The aim of this section is to establish two principles for comparing perturbed solutions of 3D Euler equations in $\bbT^3$.}
	
	\medskip
	
	Let $\bu \in L^{\infty}([0,\bar{T}];H^{s+2}(\bbT^3))$ with $s>\frac{5}{2}$ be a solution to the following Cauchy problem for the 3D incompressible Euler equations: 
	\begin{equation}\label{eq:euler bu}
		\left\{
		\begin{aligned}
			\rd_{t} \bu + \bu \cdot \nb \bu +\nb \bar{p} &= 0, \\
			\nb \cdot \bu &=0, \\
			\bu(t=0)&=\bu_0,
		\end{aligned}
		\right.
	\end{equation}
	where $\bu_0$ is a function in $H^{s+2}(\bbT^3)$. Then we consider a perturbation problem of \eqref{eq:euler bu}. To be precise, let $u$ be the solution of the following problem:
	\begin{equation}\label{eq:euler just u}
		\left\{
		\begin{aligned}
			\rd_{t} u + u \cdot \nb u +\nb p &= 0, \\
			\nb \cdot u &=0, \\
			u(t=0)&=\bu_0 + \varepsilon \tu_0,
		\end{aligned}
		\right.
	\end{equation}
	where $\varepsilon>0$ and $\tu_0$ is a function belonging to $H^{s+1}(\bbT^3)$.
	
	\medskip
	
	Now we introduce and prove our two principles.
	
	\subsection{Principle 1}
	Defining $\tu=u-\bu$ and $\tilde{p}=p-\bar{p}$, we have
	\begin{equation}\label{eq:euler tu}
		\left\{
		\begin{aligned}
			\rd_{t} \tu + \bu \cdot \nb \tu + \tu\cdot\nb\bu + \tu \cdot \nb \tu +\nb \tilde{p} &= 0, \\
			\nb \cdot \tu &=0, \\
			\tu(t=0)&=\varepsilon \tu_0.
		\end{aligned}
		\right.
	\end{equation}
	Next, we consider linearization around $\bu$: writing $\llu = \bu + \ltu$, and dropping quadratic terms in the perturbation, we arrive at
	\begin{equation}\label{eq:euler ltu}
		\left\{
		\begin{aligned}
			\rd_{t} \ltu + \bu \cdot \nb \ltu + \ltu\cdot\nb\bu +\nb \tilde{p}^{Lin} &= 0, \\
			\nb \cdot \ltu &=0, \\
			\ltu(t=0)&=\varepsilon \tu_0.
		\end{aligned}
		\right.
	\end{equation}
	Then we estimate $\normifr{\omega(t,\cdot)-\omega^{Lin}(t,\cdot)}$ on $[0,\bar{T}]$ for sufficiently small $\varepsilon>0$, where $\omega = \nb \times u$ and $\omega^{Lin} = \nb \times \llu$ are corresponding vortices.:
	\begin{proposition}\label{prop: perturbation}
		Under the above setting, there exists a constant $C>0$ such that if $\varepsilon >0$ satisfies
		\begin{equation}\label{condi: eps}
			\varepsilon \le  \frac{1}{C\normhspsr{\tu_0} \Bar{T} }\exp\left(-C\int_0^{\Bar{T}} \normhsppsr{\bu(\tau,\cdot)}d\tau \right) ,
		\end{equation}
		then on $[0,\Bar{T}]$, we have 
		\begin{equation}\label{eq:perturbation1}
			\normifr{\omega(t,\cdot)-\omega^{Lin}(t,\cdot)} \le \varepsilon^2\normhspsr{\tu_0}^2C\Bar{T}\exp\left(C\int_0^{\Bar{T}} \normhsppsr{\bu(\tau,\cdot)}d\tau \right).
		\end{equation} 
	\end{proposition}

	\begin{remark}
		The point is that the right hand side of \eqref{eq:perturbation1} is $O(\varepsilon^2)$, while a priori the left hand side is only $O(\varepsilon)$ a priori. Furthermore, the right hand side depends only on $\tu_0$ and $\bu$. 
	\end{remark}
	
	\begin{proof}
		From an elementary estimate
		\begin{equation*}
			\normifr{\omega(t,\cdot)-\omega^{Lin}(t,\cdot)} \lesssim \normhssr{\omega(t,\cdot)-\omega^{Lin}(t,\cdot)} \lesssim \normhsr{u(t,\cdot)-  { u^{Lin}} (t,\cdot)},
		\end{equation*}
		it suffices to prove 
		\begin{equation*}
			\normhsr{u(t,\cdot)-u^{Lin}(t,\cdot)} \le \varepsilon^2\normhspsr{\tu_0}^2C\Bar{T}\exp\left(C\int_0^{\Bar{T}} \normhsppsr{\bu(\tau,\cdot)}d\tau \right)
		\end{equation*}
		for $\varepsilon>0$ satisfying \eqref{condi: eps}.
		Our first step is to estimate $\normhs{\tu(t,\cdot)}$ on $[0,\bar{T}]$. Denoting $J=(I-\Delta)^{\frac{1}{2}}$, \eqref{eq:euler tu} gives
		\begin{equation*}
			\begin{split}
				\frac12 \frac{d}{dt} \normhsps{\tu}^2& = -\int J^{s+1}(\bu\cdot\nb \tu)\cdot J^{s+1}\tu
				-\int J^{s+1}(\tu\cdot\nb \bu)\cdot J^{s+1}\tu \\
				&\quad -\int J^{s+1}(\tu\cdot\nb \tu)\cdot J^{s+1}\tu   -\int J^{s+1}\nb\tilde{p}\cdot J^{s+1}\tu \\
				&=\mrI + \mrII + \mrIII + \mrIV.
			\end{split}
		\end{equation*}
		Since $\tu$ is divergence-free, $\mrIV=0$. Using the fact that $H^{s+1}{(\mathbb{T}^3)}$ is a Banach algebra, we have
		\begin{equation*}
			\mrII \lesssim \normhsps{\tu\cdot\nb \bu}\normhsps{\tu} \lesssim \normhspps{\bu}\normhsps{\tu}^2.
		\end{equation*}
		For $\mrI$, we note that $\nb \cdot \tu =0$ yields $\int \left(\bu \cdot \nb J^{s+1}\tu\right) \cdot J^{s+1} \tu =0,$
		so that
		\begin{equation*}
			\mrI = -\int \left(\left[J^{s+1}, \bu \cdot\right] \nb \tu\right) \cdot J^{s+1}\tu.
		\end{equation*}
		Recalling the Kato-Ponce commutator estimate (\cite{KP88}):
		\begin{equation*}\label{est:K-P}
			\normbr{\left[J^{s'},f\right]g} \lesssim \left(\nrm{f}_{H^{\frac{5}{2}+\varepsilon'}(\bbT^3)}\normbr{J^{s'-1}g}
			+ \normbr{J^{s'}f}\nrm{g}_{H^{\frac{3}{2}+\varepsilon'}(\bbT^3)}\right)
		\end{equation*}
		for any $\varepsilon'>0$ and $s'>0$, we obtain
		\begin{equation*}
			\mrI \lesssim \normb{\left[J^{s+1}, \bu \cdot\right] \nb \tu}\normb{J^{s+1}\tu}
			\lesssim \normhsps{\bu}\normhsps{\tu}^2.
		\end{equation*}
		Replacing $\bu$ with $\tu$ and proceeding in the same way, we can also estimate $\mrIII \lesssim \normhsps{\tu}^3.$     Combining all, we arrive at
		\begin{equation*}
			\frac{d}{dt} \normhsps{\tu} \le C \left(\normhspps{\bu}\normhsps{\tu} + \normhsps{\tu}^2\right).
		\end{equation*}
		Introducing the quantity $y(t) = \normhsps{\tu(t)}\exp(-C\int_0^t \normhspps{\bu(\tau)}d\tau ),$ 
		we have
		\begin{equation*}
			\left\{
			\begin{aligned}
				&\frac{d}{dt}y(t) \le C \exp\left(C\int_0^{\Bar{T}} \normhspps{\bu(\tau)}d\tau \right) y^2(t), \\
				&  { y(0) } =\varepsilon \normhsps{\tu_0}
			\end{aligned}
			\right.
		\end{equation*}
		on $[0,\Bar{T}]$. Then for $\varepsilon>0$ satisfying  {\eqref{condi: eps} (by adjusting $C$ if necessary)}, 
		we have
		\begin{equation*}
			y(t) \le \frac{\varepsilon\normhsps{\tu_0}}{1-\varepsilon\normhsps{\tu_0}Ct\exp\left(C\int_0^{\Bar{T}} \normhspps{\bu(\tau)}d\tau \right)} \le 2\varepsilon \normhsps{\tu_0},
		\end{equation*}
		and consequently on $[0,\Bar{T}]$, we have 
		\begin{equation}\label{est:tu}
			\normhsps{\tu(t)} \le 2\varepsilon \normhsps{\tu_0}\exp\left(C\int_0^{\bar{T}} \normhspps{\bu(\tau)}d\tau \right).
		\end{equation} Next, we consider the equation of $\du = \tu-\ltu$: 
		\begin{equation*}
			\left\{\begin{split}
				\rd_t \du + \bu \cdot \nb \du + \du \cdot \nb \bu + \tu \cdot \nb \tu + \nb (\tilde{p}-\Tilde{p}^{Lin})&=0, \\
				\nb \cdot \du &=0,\\
				\du(t=0)&=0,
			\end{split}
			\right.
		\end{equation*}
		which gives
		\begin{equation*}
			\begin{split}
				\frac12 \frac{d}{dt} \normhs{\du}^2& = -\int J^{s}(\bu\cdot\nb \du)\cdot J^{s}\du
				-\int J^{s}(\du\cdot\nb \bu)\cdot J^{s}\du  \\
    & \qquad -\int J^{s}(\tu\cdot\nb \tu)\cdot J^{s}\du   -\int J^{s}\nb(\tilde{p}-\Tilde{p}^{Lin})\cdot J^{s}\du \\
				&=\mrV + \mrVI + \mrVII + \mrVIII.
			\end{split}
		\end{equation*}
		To estimate $\mrV$, $\mrVI$, and $\mrVIII$, proceeding in the same way as $\mrI$, $\mrII$, and $\mrIV$, respectively, we have
		\begin{equation*}
			\mrV \lesssim \normhs{\bu}\normhs{\du}^2, \qquad \mrVI \lesssim \normhsps{\bu}\normhs{\du}^2,\qquad \mrVIII =0.
		\end{equation*}
		For $\mrVII$, the fact that $H^s{(\mathbb{T}^3)}$ is a Banach algebra implies
		\begin{equation*}
			\mrVII \lesssim \normhs{\tu\cdot\nb \tu} \normhs{\du} \lesssim \normhsps{\tu}^2\normhs{\du}.
		\end{equation*}
		Combining all, we obtain
		\begin{equation*}
			\frac{d}{dt} \normhs{\du} \le C \left(\normhsps{\bu}\normhs{\du} + \normhsps{\tu}^2\right).
		\end{equation*}
		Thus, for $\varepsilon>0$ satisfying \eqref{condi: eps}, the Gr\"onwall's inequality and \eqref{est:tu} give us
		\begin{equation*}
			\begin{split}
				\normhs{\du(t)} 
				&\le \exp \left(C\int_0^{t} \normhsps{\bu(\tau)}d\tau \right) C \int_0^t \normhsps{\tu(\tau)}^2 d\tau \le 4\varepsilon^2\normhsps{\tu_0}^2C\Bar{T}\exp\left(3C\int_0^{\Bar{T}} \normhspps{\bu(\tau)}d\tau \right)
			\end{split}
		\end{equation*}
		on $[0,\Bar{T}]$. Since $\du=\tu-\ltu=u-\llu$, we are done.
	\end{proof}

	\subsection{Principle 2}
	Let $u$ be the solution of \eqref{eq:euler just u}. We consider the following PDE: 
	\begin{equation}\label{eq:euler busu}
		\left\{
		\begin{aligned}
			\rd_{t} \busu + u \cdot \nb \busu +\nb \bar{p}^*-\nb \busu \cdot (u-\busu) &= 0, \\
			\nb \cdot \busu &=0, \\
			\busu(t=0)&=\bu_0,
		\end{aligned}
		\right.
	\end{equation}
	where $(\nb \busu \cdot (u-\busu))_{i}:=\sum_{j=1}^3\rd_{i}\busu_{j} (u-\busu)_{j}$ with $i=1,2,3$. (This is different from $(u-\busu)\cdot \nb\busu$.) This time, we compare $\busu$ with $\bu$ of \eqref{eq:euler bu}. Note that \eqref{eq:euler bu} and \eqref{eq:euler busu} share the same initial data. We have the following:
	\begin{proposition}\label{prop: perturbation 2}
		Under the above setting, for $t\in[0,\Bar{T}]$ and $\varepsilon>0$ satisfying \eqref{condi: eps}, there exists a constant $C>0$ such that
		\begin{equation*}
			\normhsr{\busu(t,\cdot)-\bu(t,\cdot)} \le \varepsilon\normhsps{\tu_0}C\bar{T}\left(\sup_{t\in[0,\bar{T}]}\normhspsr{\bu(t)}\right)\exp\left(C\int_0^{\Bar{T}}\left(1+ \normhsppsr{\bu(\tau)}\right)d\tau \right).
		\end{equation*}
	\end{proposition}
	\begin{proof}
		Denoting $\tu=u-\bu$ and $\dusu=\busu-\bu$, we obtain the equation of $\dusu$:
		\begin{equation*}
			\left\{\begin{split}
				&\rd_t \dusu 
				+\tu \cdot \nb \dusu + \bu \cdot \nb \dusu + \tu \cdot \nb \bu -\nb \dusu \cdot \tu + \nb \dusu \cdot \dusu - \nb \bu \cdot \tu + \nb \bu \cdot \dusu + \nb (\bar{p}^*-\bar{p})=0, \\
				&\nb \cdot \dusu =0,\\
				&\dusu(t=0)=0.
			\end{split}
			\right.
		\end{equation*}
		This gives
		\begin{equation*}
			\begin{split}
				\frac12 \frac{d}{dt} \normhs{\dusu}^2& = -\int J^{s}(\tu\cdot\nb \dusu)\cdot J^{s}\dusu
				-\int J^{s}(\bu \cdot \nb \dusu)\cdot J^{s}\dusu  -\int J^{s}(\tu \cdot \nb \bu)\cdot J^{s}\dusu  \\
				&\quad +\int J^{s}(\nb \dusu \cdot \tu)\cdot J^{s}\dusu
				-\int J^{s}(\nb \dusu \cdot \dusu)\cdot J^{s}\dusu
				+\int J^{s}(\nb \bu \cdot \tu)\cdot J^{s}\dusu \\
				&\quad -J^{s}(\nb \bu \cdot \dusu)\cdot J^{s}\dusu
				-\int J^{s}\nb(\tilde{p}-\Tilde{p}^{Lin})\cdot J^{s}\dusu \\
				&=\mrI+\mrII+\mrIII+\mrIV+\mrV + \mrVI + \mrVII + \mrVIII.
			\end{split}
		\end{equation*}
		For $\mrI$, $\mrII$, and $\mrVIII$, we proceed in the same way as the proof of Proposition \ref{prop: perturbation} to obtain
		\begin{equation*}
			\mrI \lesssim \normhs{\tu}\normhs{\dusu}^2, \qquad \mrII \lesssim \normhs{\bu}\normhs{\dusu}^2, \qquad \mrVIII =0.
		\end{equation*}
		For $\mrIII+\mrVI+\mrVII$, we use the fact that $H^{s+1}{(\mathbb{T}^3)}$ is a Banach algebra to have
		\begin{equation*}
			\begin{split}
				\mrIII+\mrVI+\mrVII 
				\lesssim \left(\normhs{\tu} + \normhs{\dusu}\right)\normhsps{\bu}\normhs{\dusu}.
			\end{split}
		\end{equation*}
		For $\mrIV$ and $\mrV$, $\nb \cdot\dusu=0$ and the integration by parts give us
		\begin{equation*}
			\mrIV= -\int J^{s}( \dusu_j \rd_i \tu_j) J^{s}\dusu_i \lesssim \normhsps{\tu}\normhs{\dusu}^2, \qquad 
			\mrV= -\frac12\int J^{s}\rd_i |\dusu|^2 J^{s}\dusu_i=0.
		\end{equation*}
		Combining all, we arrive at
		\begin{equation*}
			\frac{d}{dt} \normhs{\dusu} 
			\le C \left(\normhsps{\tu} + \normhsps{\bu} \right)\normhs{\dusu} + \normhsps{\tu} \normhsps{\bu}.
		\end{equation*}
		Thus, for $\varepsilon>0$ satisfying \eqref{condi: eps}, the Gr\"onwall's inequality and \eqref{est:tu} yield
		\begin{equation*}
			\begin{split}
				\normhs{\dusu(t)} 
				&\le \exp \left(C\int_0^{t} \left(\normhsps{\tu(\tau)} + \normhsps{\bu(\tau)} \right)d\tau \right) C \int_0^t \normhsps{\tu(\tau)} \normhsps{\bu(\tau)} d\tau\\
				&\le 2\varepsilon\normhsps{\tu_0}C\bar{T}\left(\sup_{t\in[0,\bar{T}]}\normhsps{\bu(t)}\right)\exp\left(2C\int_0^{\Bar{T}}\left(1+ \normhspps{\bu(\tau)}\right)d\tau \right)
			\end{split}
		\end{equation*}
		on $[0,\Bar{T}]$.
	\end{proof}

	\section{Proof of the main result}
	The aim of this section is to show Theorem \ref{thm: presence}. In the proof, we shall always assume that initial data $\iw_0$ and $\sw_0$ satisfy \eqref{def: iw_0} and \eqref{def: sw_0}, respectively. We note that by our definitions of $\iw$ and $\sw$, they solve the following coupled system: $ (\iw,\sw)(t=0) = (\iw_0,\sw_0)$ and
		\begin{equation}\label{eq:euler-l-vorticity-original}
			\left\{
			\begin{aligned}
				\rd_{t} \iw + (\lu + \iu+\su) \cdot \nb \iw &= \nb (\lu+\iu+\su) \iw, \\
				\rd_{t} \sw + (\lu+\iu+\su) \cdot \nb \sw &= \nb (\lu+\iu+\su) \sw, \\
				(\iu,\su) &=\nb \times (-\Dlt)^{-1}(\iw,\sw).\\
			\end{aligned}
			\right.
		\end{equation}
		Recalling the definition of $\lu$ in \eqref{eq:u-L}, we have $\lw:=\nb \times \lu=0$  {in $|x|\le 10$}. Thus, setting $\busu:=\lu+\iu$ and $u:=\lu+\iu+\su$, and noticing 
		\begin{equation*}
			\begin{split}
				\nb \times \left( u \cdot \nb \busu -\nb \busu \cdot (u-\busu)\right)
				&= \nb \times \left( \busu \cdot \nb \busu + (u-\busu) \cdot \nb \busu -\nb \busu \cdot (u-\busu)\right) \\
				&=\nb \times \left( (\bar{\omg}^* \times \busu) +\frac{1}{2}\nabla |\busu|^2 + \bar{\omg}^* \times (u-\busu)\right) \qquad (\bar{\omg}^*:= \nb \times \busu) \\
				&= \busu \cdot \nb \bar{\omg}^*- \nb \busu \bar{\omg}^* + (u-\busu) \cdot \nb \bar{\omg}^*-\nb (u-\busu) \bar{\omg}^*  =u\cdot\nb \bar{\omg}^*-\nb u \,\bar{\omg}^*,
			\end{split}
		\end{equation*}
		we can check that $\busu$ and $u$ solve \eqref{eq:euler busu} and \eqref{eq:euler just u}, respectively.
		On the other hand, we shall introduce pseudo-solutions $(\piw,\psw)$ as the solutions of
		\begin{equation}\label{eq:euler-l-vorticity-i}
			\left\{
			\begin{aligned}
				\rd_{t} \piw + (\lu + \piu) \cdot \nb \piw &= \nb (\lu+\piu) \piw, \\
				\piu &=\nb \times (-\Dlt)^{-1}\piw,
			\end{aligned}
			\right.
		\end{equation}
		and
		\begin{equation}\label{eq:euler-l-vorticity-s}
			\left\{
			\begin{aligned}
				\rd_{t} \psw + (\lu+\piu+\psu) &\cdot \nb \psw  = \nb (\lu+\piu+\psu) \psw + \nb \psu \piw -\psu\cdot \nb \piw, \\
				\psu &=\nb \times (-\Dlt)^{-1}\psw,  
			\end{aligned}
			\right.
		\end{equation} with initial data  $\piw(t=0) = \iw_0$ and $ \psw(t=0) = \sw_0$.
		Then  we can check that $\bu:=\lu+\piu$ and $u:=\lu+\piu+\psu$ are solutions of \eqref{eq:euler bu} and \eqref{eq:euler just u}, respectively. Note that $u=\lu+\piu+\psu=\lu+\iu+\su$, which implies $\piu-\iu=\su-\psu$.
		Our strategy is first to analyze pseudo-solutions $(\piw,\psw)$, and then to compare them with real solutions $(\iw,\sw)$ using the principles introduced in the previous section.  
	 
	\subsection{Behavior  of $\piw$ and $\nb \piu$}
	Abusing the notation for simplicity, we denote pseudo-solutions $\piw$ and $\piu$ by $\iw$ and $\iu$, respectively.
	We note that the the symmetry $\iw(t,x_1,x_2,x_3)=(\iw_{1}(t,x_2,x_3),0,0)$ propagates in time for the solutions to \eqref{eq:euler-l-vorticity-i}. Namely, the assumptions
	\begin{equation}\label{condi: iw}
		\rd_1 \iw_1 \equiv 0,\quad \iw_2\equiv 0, \quad \iw_3 \equiv 0
	\end{equation}
	hold for all times if they are valid at $t=0$. This is not trivial since $\lu$ depends on $x_1$. To see this, first note that \eqref{condi: iw} implies 
	\begin{equation}\label{condi: iu}
		\iu_1 \equiv 0,\quad \rd_1\iu_2\equiv 0, \quad \rd_1\iu_3 \equiv 0
	\end{equation}
	by the Biot-Savart law. Together this implies $\nb \iu \iw \equiv 0$. Denoting $D_t=\rd_t + (\lu + \iu)\cdot \nb$, from
	\begin{equation*}
		D_t(\iw_2)=(\nb\iu \iw)_2 - M\iw_2, \quad D_t(\iw_3)=(\nb\iu \iw)_3,
	\end{equation*}
	we see that $(\nb \iu \iw)\equiv 0$ is consistent with $\iw_2,\iw_3$ being zero for all times. Lastly,
	\begin{equation*}
		\begin{aligned}
			D_t(&\iw_1)=(\nb\iu \iw)_1 + M\iw_1,\qquad D_t(\rd_1\iw_1)=\rd_1((\nb\iu \iw)_1),
		\end{aligned}
	\end{equation*}
	which shows that $\rd_1\iw_1 \equiv 0$ propagates in time.
	This shows that the equation for $\iw_1$ is given by
	\begin{equation}\label{eq:reduced iw}
		\rd_t \iw_1 + (\iu_2-Mx_2)\rd_2\iw_1 +\iu_3\rd_3\iw_1 = M\iw_1.
	\end{equation}
	Comparing \eqref{eq:reduced iw} with 2D Euler equation which has global well-posedness of smooth solutions, we can check that there exists the unique global smooth solution $\iw_1$ of \eqref{eq:reduced iw} with initial data given in \eqref{def: iw_0}.  Furthermore, we can also observe that $\iw_1$ keeps the odd symmetry in both $x_2$ and $x_3$.
	
	To begin with, we observe temporal behaviors of $\normp{\iw(t,\cdot)}$ for $p\in[1,\infty]$.
	\begin{lemma}
		For $t\in[0,\infty)$, $\iw(t,\cdot)=(\iw_1(t,\cdot),0,0)$ satisfies
		\begin{equation}\label{est: iw}
			\normp{\iw_1(t,\cdot)}  { = \normp{\iw_{0}}}e^{\frac{M(p-1)}{p}t}\;(1\le p <\infty)\quad \text{and} \quad \normif{\iw_1(t,\cdot)}  {=\normif{\iw_{0}}} e^{Mt}.
		\end{equation}
	\end{lemma}
	\begin{proof}
		Taking the $L^2$ inner product of \eqref{eq:reduced iw} with $\iw_1|\iw_1|^{p-2}$, we have
		\begin{equation*}
			\begin{split}
				\frac{1}{p}\frac{d}{dt}\normp{\iw_1}^p &= -\int (\iu_2\rd_2\iw_1 + \iu_3\rd_3\iw_1)\iw_1|\iw_1|^{p-2}
				+\int Mx_2\rd_2\iw_1\iw_1|\iw_1|^{p-2} + M\normp{\iw_1}^p.
			\end{split}
		\end{equation*}
		After the integration by parts, the first integral vanishes since $\iu_1=0$ and $\nb \cdot \iu =0$, and the second integral is equal to $-\frac{M}{p}\normp{\iw_1}^p$. Thus we have
		\begin{equation*}
			\begin{split}
				\frac{d}{dt}\normp{\iw_1}^p &= M(p-1)\normp{\iw_1}^p, \quad \mbox{    which implies}         \qquad    \normp{\iw_1(t,\cdot)}  = e^{\frac{M(p-1)}{p}t}\normp{\iw_{1,0}}.
			\end{split}
		\end{equation*} 
		Passing $p\rightarrow \infty$, we can also obtain $L^{\infty}$ estimate.
	\end{proof}
	
	Next, using \eqref{est: iw}, we estimate $\nrmb{\iu(t,\cdot)}_{H^s}$ with $s>\frac{5}{2}$ for $t\in[0,T_M]$ with $T_M=\frac{\log(M+1)}{M}$. This estimate will be used in Section \ref{comparison}.
	\begin{lemma}\label{lem: iu-Hs-bound}
		For $s>\frac52$ and $t\in[0,T_M]$, $\normhsr{\iu(t,\cdot)}\lesssim e^{ {C}Mt}$.
	\end{lemma}
	\begin{proof}
		Fix $s>\frac{5}{2}$.
		Noticing that $\iu$ solves
		\begin{equation*}
			\left\{
			\begin{aligned}
				\rd_t \iu + (\lu + \iu) \cdot \nb \iu + \nb p^{\mathcal{I}} = 0 ,  \\
				\nb\cdot \iu = 0.
			\end{aligned}
			\right.
		\end{equation*}
		and denoting $J=(I-\Delta)^{\frac{1}{2}}$, we have
		\begin{equation*}
			\begin{split}
				\frac12 \frac{d}{dt} \normhs{\iu}^2& = -\int J^{s}(\lu\cdot\nb \iu)\cdot J^{s}\iu
				-\int J^{s}(\iu\cdot\nb \iu)\cdot J^{s}\iu -\int J^{s}\nb p^{\mathcal{I}}\cdot J^{s}\iu =\mrI + \mrII + \mrIII.
			\end{split}
		\end{equation*}
		 {$\mrIII=0$ follows from $\nb \cdot \iu=0$.}  {Noticing} $ \int \left(\lu \cdot \nb J^{s}\iu\right) \cdot J^{s} \iu= \int \left(\iu \cdot \nb J^{s}\iu\right) \cdot J^{s} \iu =0,$
		we obtain
		\begin{equation*}
			\mrI = -\int \left(\left[J^{s}, \lu \cdot\right] \nb \iu\right) \cdot J^{s}\iu
			\quad \text{and} \quad \mrII = -\int \left(\left[J^{s}, \iu \cdot\right] \nb \iu\right) \cdot J^{s}\iu, 
		\end{equation*} where we recall that $[\cdot, \cdot]$ denotes the commutator. Using $\nrmb{\lu}_{H^{s}(\bbT^3)}\lesssim M$ and the Sobolev embedding $H^s(\bbT^3) \hookrightarrow W^{1,\infty}(\bbT^3)$, we obtain
		\begin{equation*}
			\mrI \lesssim \normb{\left[J^{s}, \lu \cdot\right] \nb \iu}\normb{J^{s}\iu}
			\lesssim M\normhs{\iu}^2, \quad \mrII \lesssim \normb{\left[J^{s}, \iu \cdot\right] \nb \iu}\normb{J^{s}\iu}
			\lesssim \normif{\nabla\iu}\normhs{\iu}^2,
		\end{equation*}
		 {which lead to}
		\begin{equation*}
			\frac{d}{dt} \normhs{\iu} \lesssim \left(M+ \normif{\nabla\iu}\right)\normhs{\iu}.
		\end{equation*}
		According to the Calderon-Zygmund theory, we have
		\begin{equation*}
			\normifr{\nabla\iu} \lesssim \normifr{\iw}\log \left(10 + \frac{\nrmb{\iu}_{H^s(\bbT^3)}}{\normifr{\iw}} \right)\lesssim e^{Mt}\log \left(10 + \normhs{\iu} \right),
		\end{equation*}
		where we used \eqref{est: iw} in the last inequality. 
		This gives
		\begin{equation*}
			\frac{d}{dt}\log \left(10 + \normhs{\iu}\right) \lesssim M+ e^{Mt}\log \left(10 + \normhs{\iu}\right).
		\end{equation*}
		 {Using Gr\"onwall's inequality, we obtain the desired estimate on $[0,T_M]$.}
	\end{proof}

	Our next aim is to estimate the maximum of $\nb \iu(t,\cdot)$ in a small region near the origin up to time $T_M=\frac{\log(M+1)}{M}$.
	The corresponding velocity $\iu=(0,\iu_2,\iu_3)$ to $\iw=(\iw_1,0,0)$ has explicit formula:
	\begin{equation}\label{formula: iu}
		\begin{split}
			\iu_2(t,x_2,x_3)=\frac{1}{2\pi}\sum_{n=(n_2,n_3)\in \mathbb{Z}^2}\iint_{[-L,L)^2} \frac{-x_{3}+y_3+2Ln_3}{|x-y-2Ln|^2}\iw_1(t,y_2,y_3)dy_2dy_3, \\
			\iu_3(t,x_2,x_3)=\frac{1}{2\pi}\sum_{n=(n_2,n_3)\in \mathbb{Z}^2}\iint_{[-L,L)^2} \frac{x_2-y_2-2Ln_2}{|x-y-2Ln|^2}\iw_1(t,y_2,y_3)dy_2dy_3.
		\end{split}
	\end{equation}
	Using this, we can prove a log-Lipschitz estimate of $\iu$.
	\begin{lemma}\label{lemma: iu log-lipschitz}
		Let $x,x'\in \mathbb{T}^2:=[-L,L)^2$. Then we have
		\begin{equation}\label{est: iu}
			\left|\iu(t,x)-\iu(t,x')\right| \lesssim  e^{tM}|x-x'|\left(1+\log \frac{3L}{|x-x'|} \right).
		\end{equation}
	\end{lemma} Note that the argument of the logarithm in \eqref{est: iu} is always greater than 1 because $|x-x'|\le 2\sqrt{2}L$. We omit the proof since it follows directly from the standard log-Lipschitz estimate for 2d Euler (see for instance \cite{MB}).

	Now we consider a characteristic curve $A^{\mathcal{I}}(t,a_2,a_3)=\left(A^{\mathcal{I}}_2(t,a_2,a_3),A^{\mathcal{I}}_3(t,a_2,a_3)\right):[0,\infty)\times \mathbb{T}^2\rightarrow \mathbb{T}^2$ of \eqref{eq:reduced iw} defined by $A^{\mathcal{I}}(0,a_2,a_3)=(a_2,a_3)$ and 
	\begin{equation}\label{def: iA}
		\left\{
		\begin{aligned}
			&\frac{d}{dt}A^{\mathcal{I}}_2(t,a_2,a_3)=\iu_2(t,A^{\mathcal{I}}(t,a_2,a_3))-MA^{\mathcal{I}}_2(t,a_2,a_3),\\
			&\frac{d}{dt}A^{\mathcal{I}}_3(t,a_2,a_3)=\iu_3(t,A^{\mathcal{I}}(t,a_2,a_3)).
		\end{aligned}
		\right.
	\end{equation}
	Evaluating along this curve, we have from \eqref{eq:reduced iw}
	\begin{equation}\label{value: iw}
		\iw_1(t,\iA_2(t,a),\iA_3(t,a))=e^{Mt}\iw_{1,0}(a_2,a_3). 
	\end{equation}
	We need the following two lemmas for $A^{\mathcal{I}}$.
	\begin{lemma}\label{determinant of iA}
		The determinant of $\nb_a\iA$ with $a=(a_2,a_3)$ satisfies
		\begin{equation}\label{eq:detA}
			\det (\nb_a\iA) = e^{-Mt}.
		\end{equation}
	\end{lemma}
	\begin{proof}
		Using $\iu_1=0$ and $\nb \cdot \iu =0$, we compute 
		\begin{equation*}
			\begin{aligned}
				\frac{d}{dt}\det (\nb_a\iA)(t,a)&=(\rd_2 \iu_2(t,A^{\mathcal{I}}(t,a)) -M + \rd_3\iu_3(t,A^{\mathcal{I}}(t,a)))\det (\nb_a\iA)(t,a) = -M\det (\nb_a\iA)(t,a).
			\end{aligned}
		\end{equation*}
		Since $\det (\nb_a\iA)(0,a)=1$, we obtain \eqref{eq:detA}.
	\end{proof}
	
	\begin{lemma}\label{lemma: iA}
		Let $1\le a_2,a_3\le 2$. Then for $0\le t\le T_M$, there exist constants $C_1,\,C_2 >0$, $C_3\ge1$ independent of $M$ such that
		\begin{equation}\label{est:A_2}
			C_1 e^{-C_3 Mt}\le \iA_2(t,a_2,a_3) \le C_2 e^{-C_3^{-1} Mt},
		\end{equation}
		and
		\begin{equation}\label{est:A_3}
			C_1 \le \iA_3(t,a_2,a_3) \le C_2.
		\end{equation}
	\end{lemma}
	\begin{proof}
		Note that $\iu_2(t,0,\iA_3(t,a_2,a_3))=0$ for all $t$ by odd symmetry of $\iw_1$, which gives
		\begin{equation}\label{equality: iA_2}
			\frac{d}{dt}A^{\mathcal{I}}_2(t,a_2,a_3)=\iu_2(t,A^{\mathcal{I}}(t,a_2,a_3))-\iu_2(t,0,A_3^{\mathcal{I}}(t,a_2,a_3))-MA^{\mathcal{I}}_2(t,a_2,a_3).
		\end{equation}
		To begin with, we show the upper bound of $\iA_2(t,a_2,a_3)$ in \eqref{est:A_2}.  
		Applying \eqref{est: iu} and noticing $A^{\mathcal{I}}_2(t,a_2,a_3)>0$, we have
		\begin{equation*}
			\frac{d}{dt}A^{\mathcal{I}}_2(t,a_2,a_3)
			\le C e^{tM}
			A_2^{\mathcal{I}}(t,a_2,a_3)\left(1+\log \frac{3L}{A_2^{\mathcal{I}}(t,a_2,a_3)}\right)-MA^{\mathcal{I}}_2(t,a_2,a_3),
		\end{equation*}
		which leads to
		\begin{equation*}
			\frac{d}{dt}\left(\log A^{\mathcal{I}}_2(t,a_2,a_3)-(1+\log 3L)\right)\le -C e^{tM}\left(\log A^{\mathcal{I}}_2(t,a_2,a_3)-(1+\log 3L)\right)-M.
		\end{equation*}
		From
		\begin{equation*}
			\frac{d}{dt}\left(\left(\log A^{\mathcal{I}}_2(t,a_2,a_3)-(1+\log 3L)\right)\exp{\left( \int_{0}^{t} Ce^{sM} ds \right)} \right) \le -M\exp{\left( \int_{0}^{t} Ce^{sM} ds \right)} \le -M,
		\end{equation*}
		we see that for $0\le t \le T_M$,
   {
		\begin{equation}\label{est: A_2 sample}
			\begin{split}
				\log A^{\mathcal{I}}_2(t,a_2,a_3) 
				&\le 1+\log 3L + \exp\left(-\int_{0}^{t}Ce^{sM}ds\right) \left(\log a_2 -(1+\log 3L) -Mt\right) \\
				&\le 1+\log 3L+ \exp(-C)\left(\log a_2 -(1+\log3L)-Mt \right),
			\end{split}
		\end{equation}
  } where we used $\log a_2 -(1+\log 3L)-Mt<0$ and $t\le T_M=\frac{\log(M+1)}{M}$.
		Since $a_2\le 2$, we arrive at
		\begin{equation*}
			A^{\mathcal{I}}_2(t,a_2,a_3) \lesssim \exp \left(\exp(-C)(\log 2 -(1+\log 3L) -Mt)\right)
		\end{equation*}
		for $t\in[0,T_M]$. We take $C_3:=\exp(C)$, which satisfies $C_3\ge1$.
		
		Now we prove the lower bound of $\iA_2(t,a_2,a_3)$ in \eqref{est:A_2}. Recalling \eqref{equality: iA_2} and \eqref{est: iu}, we have
		\begin{equation*}
			\begin{split}
				-\frac{d}{dt}A^{\mathcal{I}}_2(t,a_2,a_3)
				&=-\iu_2(t,A^{\mathcal{I}}(t,a_2,a_3))+\iu_2(t,0,A_3^{\mathcal{I}}(t,a_2,a_3))+MA^{\mathcal{I}}_2(t,a_2,a_3) \\
				&\le C e^{tM}
				A_2^{\mathcal{I}}(t,a_2,a_3)\left(1+\log \frac{3L}{A_2^{\mathcal{I}}(t,a_2,a_3)}\right)+MA^{\mathcal{I}}_2(t,a_2,a_3),
			\end{split}
		\end{equation*}
		which yields
		\begin{equation*}
			\frac{d}{dt}\left(1+\log \frac{3L}{A_2^{\mathcal{I}}(t,a_2,a_3)}\right)
			\le C e^{tM}
			\left(1+\log \frac{3L}{A_2^{\mathcal{I}}(t,a_2,a_3)}\right)+M.
		\end{equation*}
		Noticing
		\begin{equation*}
			\frac{d}{dt}\left(\left(1+\log \frac{3L}{A_2^{\mathcal{I}}(t,a_2,a_3)}\right)\exp{\left(- \int_{0}^{t} C e^{sM}ds \right)} \right) \le M\exp{\left( -\int_{0}^{t} C e^{sM} ds \right)} \le M,
		\end{equation*}
		we see that for $0\le t \le T_M$,
   {
		\begin{equation}\label{est: A_2 opposite sample}
			\begin{split}
				\log \frac{3L}{A_2^{\mathcal{I}}(t,a_2,a_3)}&\le -1+ \exp\left(\int_{0}^{t} C e^{sM}ds\right)\left(1+\log \frac{3L}{a_2} + Mt \right) 
				\le -1+ C_3\left(1+\log \frac{3L}{a_2}+Mt \right)
			\end{split}
		\end{equation} 
  }where $C_3:=\exp(C)$.
		Since $a_2\ge1$, we have $A^{\mathcal{I}}_2(t,a_2,a_3) \ge 3L\exp \left(1-C_3(1+\log 3L +Mt) \right).$ 
		To obtain the bounds of $\iA_3(t,a_2,a_3)$ in \eqref{est:A_3}, we note that $\iu_3(t,\iA_2(t,a_2,a_3),0)=0$ (by the odd symmetry of $\iw_1$) and $\iA_3(t,a_2,a_3)>0$ for all $t>0$, and proceed as we have shown the case of $\iA_2(t,a_2,a_3)$.
	\end{proof}
	
	Henceforth, let $C_i$ $(i=1,2,3)$ denote the constants in Lemma \ref{lemma: iA}. 
	Now we are ready to estimate $\normif{\nb \iu(t,\cdot)}$ near the origin.
	\begin{lemma}\label{lemma: nb iu}
		Let $\dlt:=\frac{C_1}{10(M+1)^{C_3}}$.  Then for $t\in [0,T_M]$,
		we have
		\begin{equation}\label{est: nabla iu}
			\nrmb{\nabla \iu(t,\cdot)}_{L^{\infty}(B_0(\dlt))} \lesssim e^{-C_3^{-1}Mt}.
		\end{equation}
	\end{lemma}
	\begin{remark}\label{rmk: dlt}
		We note that $\dlt \le \frac{C_1 e ^{-C_3Mt}}{10}$ for $0\le t \le T_M$ by the definitions of $\dlt$ and $T_M$.
	\end{remark}
	\begin{proof}
		Let $x=(x_2,x_3)\in B_0(\dlt)$. 
		We recall explicit formulas
		\begin{equation}\label{formula: partial2 iu2}
			\begin{split}
				&\rd_2 \iu_2 (t,x_2,x_3)= -\rd_3 \iu_3(t,x_2,x_3)\\
				&\qquad=\frac{1}{\pi}\sum_{n=(n_2,n_3)\in \mathbb{Z}^2}p.v.\iint_{[-L,L)^2} \frac{(x_2-y_2-2Ln_2)(x_3-y_3-2Ln_3)}{|x-y-2Ln|^4}\iw_1(t,y_2,y_3)dy_2dy_3 \\
			\end{split}
		\end{equation}
		Moreover, since $\iw_1=0$ in $[0,T_M]\times B_0(\dlt)$ by Lemma \ref{lemma: iA},
		\begin{equation}\label{formula: partial3 iu2}
			\begin{split}
				&\rd_2 \iu_3 (t,x_2,x_3)
				= -\rd_3 \iu_2(t,x_2,x_3) \\
				&\qquad=\frac{1}{2\pi}\sum_{n=(n_2,n_3)\in \mathbb{Z}^2}p.v.\iint_{[-L,L)^2} \frac{(x_2-y_2-2Ln_2)^2-(x_3-y_3-2Ln_3)^2}{|x-y-2Ln|^4}\iw_1(t,y_2,y_3)dy_2dy_3 
			\end{split}
		\end{equation}
		in $[0,T_M]\times B_0(\dlt)$. (See e.g. \cite{MB} for derivations of \eqref{formula: partial2 iu2} and \eqref{formula: partial3 iu2}.) 
		To begin with, we estimate $\nrmb{\rd_2\iu_2}_{L^{\infty}(B_0(\dlt))} =\nrmb{\rd_3\iu_3}_{L^{\infty}(B_0(\dlt))} $.
		
		By the odd symmetry of $\iw_1$ in both $x_2$ and $x_3$, \eqref{formula: partial2 iu2} yields
		 {
			\begin{equation*}
				\begin{split}
					&\rd_2 \iu_2 (t,x_2,x_3) \\
					&\approx \sum_{n \in \mathbb{Z}^2}\iint_{\left\{0\le y_2,y_3\le L\right\}} 
					\left[\frac{(x_2-y_2-2Ln_2)(x_3-y_3-2Ln_3)}{|x-y-2Ln|^4}
					-\frac{(x_2+y_2-2Ln_2)(x_3-y_3-2Ln_3)}{\left((x_2+y_2-2Ln_2)^2 + (x_3-y_3-2Ln_3)^2 \right)^2}
					\right. \\
					&\left.\quad 
					-\frac{(x_2-y_2-2Ln_2)(x_3+y_3-2Ln_3)}{\left((x_2-y_2-2Ln_2)^2 + (x_3+y_3-2Ln_3)^2 \right)^2}+\frac{(x_2+y_2-2Ln_2)(x_3+y_3-2Ln_3)}{\left((x_2+y_2-2Ln_2)^2 + (x_3+y_3-2Ln_3)^2 \right)^2}
					\right]\iw_1(t,y_2,y_3)dy_2y_3 \\
					&= \mrI + \mrII + \mrIII + \mrIV.
				\end{split}
			\end{equation*}
		}
		 {To estimate $\mrI$, 
			we divide it into $\mrI_0$ and $\mrI_{\neq 0}$ which correspond to the term with $n=(0,0)$ and the sum of all other terms with $n\neq (0,0)$, respectively.
			For $\mrI_0$,}
		we make a change of variables $(y_2,y_3)\mapsto (a_2,a_3)$ by $y=\iA(t,a)$ and use \eqref{eq:detA} to have
		\begin{equation*}
			 {\mrI_0} \approx \iint_{\left\{0\le a_2,a_3\le L\right\}}
			\frac{(x_2-\iA_2(t,a))(x_3-\iA_3(t,a))}{|x-\iA(t,a)|^4}\iw_1(t,\iA_2(t,a),\iA_3(t,a))e^{-Mt}da_2da_3.
		\end{equation*}
		By \eqref{value: iw} and the assumption that $\iw_{1,0}$ is supported on $\left\{1\le a_2,a_3\le 2 \right\}$ (see \eqref{def: iw_0}), we obtain
		\begin{equation*}
			 {\mrI_0} \le \normif{\iw_{1,0}} \iint_{\left\{1\le a_2,a_3\le 2\right\}}
			\frac{\left|(x_2-\iA_2(t,a))(x_3-\iA_3(t,a))\right|}{|x-\iA(t,a)|^4}da_2da_3.
		\end{equation*}
		Since $x\in B_0(\dlt)$, Lemma \ref{lemma: iA} and Remark \ref{rmk: dlt} imply
		\begin{equation}\label{est: x_2,x_3-iA_2,3}
			\left\{ \begin{aligned}
				\frac{9C_1e^{-C_3Mt}}{10} \le \iA_2(t,a)-|x_2| &\le \iA_2(t,a)-x_2 \le \iA_2(t,a)+ |x_2| \le \frac{11C_2e^{-C_3^{-1}Mt}}{10}, \\
				\frac{9C_1}{10}\le \iA_3(t,a)-|x_3| &\le\iA_3(t,a)-x_3 \le \iA_3(t,a)+|x_3| \le \frac{11C_2}{10}.
			\end{aligned}
			\right.  
		\end{equation}
		Consequently,
		\begin{equation*}
			\frac{\left|(x_2-\iA_2(t,a))(x_3-\iA_3(t,a))\right|}{|x-\iA(t,a)|^4} \le \frac{\left|(x_2-\iA_2(t,a))(x_3-\iA_3(t,a))\right|}{(x_3-\iA_3(t,a))^4}
			\lesssim e^{-C_3^{-1}Mt},
		\end{equation*}
		which gives $ {\mrI_0} \lesssim e^{-C_3^{-1}Mt}$.
		To estimate  {$\mrI_{\neq 0}$}, denoting $\tilde{n}:=(-n_2,n_3)$, we compute
		 {
			\begin{equation*}
				\begin{split}
					&\frac{(x_2-y_2-2Ln_2)(x_3-y_3-2Ln_3)}{|x-y-2Ln|^4} +\frac{(x_2-y_2+2Ln_2)(x_3-y_3-2Ln_3)}{|x-y-2L\tilde{n}|^4} \\
					&=\frac{(x_3-y_3-2Ln_3)\left((x_2-y_2)(|x-y-2L\tilde{n}|^4+|x-y-2Ln|^4)-2L_2(|x-y-2L\tilde{n}|^4-|x-y-2Ln|^4)\right)}{|x-y-2Ln|^4|x-y-2L\tilde{n}|^4}.
				\end{split}
			\end{equation*}
		}
		Since $|x-y-2L\tilde{n}|^4-|x-y-2Ln|^4=16Ln_2(x_2-y_2)\left((x_2-y_2)^2+4L^2n_2^2) + (x_3-y_3-2Ln_3)^2\right) $
		and $|x-y|\le |x|+|y|\le \dlt + L$ implies 
		\begin{equation}\label{lowerbound: x-y-2Ln}
			|x-y-2Ln|,\,|x-y-2L\Tilde{n}|\gtrsim |n|,
		\end{equation}
		we have
		\begin{equation*}
			\left|\frac{(x_2-y_2-2Ln_2)(x_3-y_3-2Ln_3)}{|x-y-2Ln|^4} +\frac{(x_2-y_2+2Ln_2)(x_3-y_3-2Ln_3)}{|x-y-2L\tilde{n}|^4}\right| \lesssim\frac{ {|x_2-y_2|}}{|n|^3}.
		\end{equation*}
		Hence, making again the change of variables $(y_2,y_3)\mapsto (a_2,a_3)$ by $y=\iA(t,a)$ and using \eqref{est: x_2,x_3-iA_2,3}, we proceed in the same argument to obtain
		$ {\mrI_{\neq 0}} \lesssim e^{-C_3^{-1}Mt}.$
		 {In a similar manner, we can also show that all of $\mrII,\, \mrIII,$ and $\mrIV$ have the same upper bound, which implies $\nrmb{\rd_2\iu_2}_{L^{\infty}(B_0(\dlt))} \lesssim e^{-C_3^{-1}Mt}.$}
		
		Next, we estimate $\nrmb{\rd_3\iu_2}_{L^{\infty}(B_0(\dlt))}\left(=\nrmb{\rd_2\iu_3}_{L^{\infty}(B_0(\dlt))}\right)$.
		By the odd symmetry of $\iw_1$ in both $x_2$ and $x_3$, \eqref{formula: partial3 iu2} yields
		 {
			\begin{equation*}
				\begin{split}
					&\rd_3 \iu_2 (t,x_2,x_3)\\
					&\approx\sum_{n \in \mathbb{Z}^2}\iint_{\left\{0\le y_2,y_3\le L\right\}} 
					\left[\frac{(x_2-y_2-2Ln_2)^2-(x_3-y_3-2Ln_3)^2}{|x-y-2Ln|^4}
					-\frac{(x_2+y_2-2Ln_2)^2-(x_3-y_3-2Ln_3)^2}{\left((x_2+y_2-2Ln_2)^2 + (x_3-y_3-2Ln_3)^2 \right)^2} \right. \\
					&\left.\quad 
					-\frac{(x_2-y_2-2Ln_2)^2-(x_3+y_3-2Ln_3)^2}{\left((x_2-y_2-2Ln_2)^2 + (x_3+y_3-2Ln_3)^2 \right)^2}+\frac{(x_2+y_2-2Ln_2)^2-(x_3+y_3-2Ln_3)^2}{\left((x_2+y_2-2Ln_2)^2 + (x_3+y_3-2Ln_3)^2 \right)^2}
					\right]\iw_1(t,y_2,y_3)dy_2y_3 \\
					&= \mrI + \mrII + \mrIII + \mrIV.
				\end{split}
			\end{equation*}
		}
		for $(t,x)\in [0,T_M] \times B_0(\dlt)$.
		To bound $\mrI +\mrII$, 
		 { 
			we again divide it into $(\mrI+\mrII)_0$ and $(\mrI+\mrII)_{\neq 0}$ which correspond to the term with $n=(0,0)$ and the sum of all other terms with $n\neq (0,0)$, respectively.
			For $(\mrI+\mrII)_0$,}
		we again make a change of variables $(y_2,y_3)\mapsto (a_2,a_3)$ by $y=\iA(t,a)$ and use \eqref{eq:detA} to have
		\begin{equation*}
			\begin{split}
				 {(\mrI +\mrII)_0}= \frac{1}{2\pi}
				\iint_{\left\{0\le a_2,a_3\le L\right\}}
				K_{x_2,x_3}(\iA_2(t,a),\iA_3(t,a))\,\iw_1(t,\iA_2(t,a),\iA_3(t,a))e^{-Mt}da_2da_3,
			\end{split}
		\end{equation*}
		where
		\begin{equation*}
			K_{x_2,x_3}(z_2,z_3) = \frac{\left((x_2-z_2)^2-(x_3-z_3)^2\right)\left((x_2+z_2)^2 + (x_3-z_3)^2 \right)^2-\left( (x_2+z_2)^2-(x_3-z_3)^2\right)|x-z|^4}{|x-z|^4\left((x_2+z_2)^2 + (x_3-z_3)^2 \right)^2}
		\end{equation*}
		for $z=(z_2,z_3)$.
		From \eqref{value: iw} and the assumption that $\iw_{1,0}$ is supported on $\left\{1\le a_2,a_3\le 2 \right\}$, we have
		\begin{equation*}
			 {(\mrI +\mrII)_0} \le \normif{\iw_{1,0}} \iint_{\left\{1\le a_2,a_3\le 2\right\}}
			\left| K_{x_2,x_3}(\iA_2(t,a),\iA_3(t,a))  \right| da_2da_3.
		\end{equation*}
		Note that the numerator of $K_{x_2,x_3}(z_2,z_3)$ can be written as
		\begin{equation}\label{est: kernel K}
			\sum_{\substack{\alp_1,\alp_2,\alp_4\ge0, \alp_3\ge1 \\ \alp_1+\alp_2+\alp_3+\alp_4=6, \alp_3\ge1}} C_{(\alp_1,\alp_2,\alp_3,\alp_4)}
			x_2^{\alp_1}x_3^{\alp_2}z_2^{\alp_3}z_3^{\alp_4}
		\end{equation}
		for some constants $C_{(\alp_1,\alp_2,\alp_3,\alp_4)}$'s.
		Moreover, the denominator of $K_{x_2,x_3}(z_2,z_3)$ is bounded below by $(x_3-z_3)^8$.
		Thus, using Lemma \ref{lemma: iA} and Remark \ref{rmk: dlt}, we have
		\begin{equation*}
			\begin{split}
				\left|K_{x_2,x_3}(\iA_2(t,a),\iA_3(t,a))\right|
				&\lesssim 	\sum_{\substack{\alp_1,\alp_2,\alp_4\ge0, \alp_3\ge1 \\ \alp_1+\alp_2+\alp_3+\alp_4=6, \alp_3\ge1}} \frac{\left|x_2^{\alp_1}x_3^{\alp_2}(\iA_2)^{\alp_3}(\iA_3)^{\alp_4}\right|}{(x_3-\iA_3(t,a))^8} \\
				&\lesssim 	\sum_{\substack{\alp_1,\alp_2,\alp_4\ge0, \alp_3\ge1 \\ \alp_1+\alp_2+\alp_3+\alp_4=6, \alp_3\ge1}}
				\left(e^{-C_3Mt}\right)^{\alp_1}\left(e^{-C_3Mt}\right)^{\alp_2}
				\left(e^{-C_3^{-1}Mt}\right)^{\alp_3} \lesssim e^{-C_3^{-1}Mt}
			\end{split}
		\end{equation*}
		for $t\in[0,T_M]$. In the last line, we used $\alp_3\ge1$. This gives $ {(\mrI +\mrII)_0} \lesssim e^{-C_3^{-1}Mt}$.
		For  {$(\mrI+\mrII)_{\neq 0}$}, we recall \eqref{est: kernel K} and \eqref{lowerbound: x-y-2Ln}, which give
		\begin{equation*}
			\begin{split}
				\left|K_{x_2-2Ln_2,x_3-2Ln_3}(y_2,y_3)\right|
				&\lesssim	\sum_{\substack{\alp_1,\alp_2,\alp_4\ge0, \alp_3\ge1 \\ \alp_1+\alp_2+\alp_3+\alp_4=6, \alp_3\ge1}} \frac{\left|(x_2-2Ln_2)^{\alp_1}(x_3-2Ln_3)^{\alp_2}y_2^{\alp_3}y_3^{\alp_4}\right|}{|n|^8} \lesssim 
				\frac{y_2}{|n|^3},
			\end{split}
		\end{equation*}
		where in the last inequality, we used $\alp_1+\alp_2\le 5$ and $\alp_3 \ge 1$. Hence proceeding as before, we obtain  {$(\mrI+\mrII)_{\neq 0}\lesssim e^{-C_3^{-1}Mt}.$}
		 {
			In the same way, we can also show that $\mrIII + \mrIV \lesssim e^{-C_3^{-1}Mt}$,
			and consequently, we obtain $\nrmb{\rd_3\iu_2}_{L^{\infty}(B_0(\dlt))} \lesssim e^{-C_3^{-1}Mt}.$
		}
	\end{proof}

	\subsection{Linearization of the equation for $\psw$}\label{section: linearization}
	Abusing the notation as in the last section, we denote pseudo-solutions $\psw$ and $\psu$ by $\sw$ and $\su$, respectively.
	Dropping nonlinear terms in \eqref{eq:euler-l-vorticity-s}, we obtain the following linearized equation of $\sw$:
	\begin{equation}\label{eq: linearized sw with iw}
		\rd_t\sw + (\lu + \iu)\cdot \nb \sw = \nb (\lu + \iu) \sw + \nb \su \iw -\su\cdot \nb \iw.
	\end{equation}
	Now recalling \eqref{def: iA} and abusing the notation, we denote a characteristic curve $$A^{\mathcal{I}}(t,a_1,a_2,a_3)=\left(A^{\mathcal{I}}_1(t,a_1,a_2,a_3),A^{\mathcal{I}}_2(t,a_1,a_2,a_3),A^{\mathcal{I}}_3(t,a_1,a_2,a_3)\right):[0,\infty)\times \mathbb{T}^3\rightarrow \mathbb{T}^3$$   defined by $A^{\mathcal{I}}(0,a_1,a_2,a_3)=(a_1,a_2,a_3)$ and 
	\begin{equation}\label{def: iA-extended}
		\left\{
		\begin{aligned}
			&\frac{d}{dt}A^{\mathcal{I}}_1(t,a_1,a_2,a_3)=MA^{\mathcal{I}}_1(t,a_1,a_2,a_3),\\
			&\frac{d}{dt}A^{\mathcal{I}}_2(t,a_1,a_2,a_3)=\iu_2(t,A_2^{\mathcal{I}}(t,a_1,a_2,a_3),A_3^{\mathcal{I}}(t,a_1,a_2,a_3))-MA^{\mathcal{I}}_2(t,a_1,a_2,a_3),\\
			&\frac{d}{dt}A^{\mathcal{I}}_3(t,a_1,a_2,a_3)=\iu_3(t,A_2^{\mathcal{I}}(t,a_1,a_2,a_3),A_3^{\mathcal{I}}(t,a_1,a_2,a_3)).
		\end{aligned}
		\right.
	\end{equation}
	Then we can show the following.
	\begin{lemma}\label{lemma: iA-small}
		Let $\dlt =\frac{C_1}{10(M+1)^{C_3}}$ as in Lemma \ref{lemma: nb iu}. Then for $a=(a_1,a_2,a_3)\in \mathbb{T}^3$ and $\ell \le  \min\left\{\frac{\dlt}{M+1}, (3Le)^{1-C_3}\dlt^{C_3}\right\}$,
		the following statements hold:
		\begin{itemize}
			\item if $a\in B_{0}(\ell)$, then $\iA(t,a)\in B_{0}(\dlt)$ for $t\in[0,T_M]$,
			\item if $a\in \mathbb{T}^3\backslash B_{0}(\ell)$, then $\iA(t,a)\in  \mathbb{T}^3\backslash B_{0}\left(C_4\left(\frac{\ell}{M+1}\right)^{C_3}\right)$ for $t\in[0,T_M]$, where $C_4>0$ is a constant.
		\end{itemize}
	\end{lemma}
	\begin{proof}
		Let us prove the first statement. Suppose that $a\in B_{0}(\ell)$. Then $|a_1|\le \ell$ implies that 
		$\left|\iA_1(t,a_1,a_2,a_3)\right| = |a_1| e^{Mt}\le \ell(M+1) \le \dlt$ for $t\in[0,T_M]$. For $\iA_2(t,a)$ and $\iA_3(t,a)$, we only need to consider the case when $a_2,a_3\ge 0$ by the odd symmetry of $\iw$ in both $x_2$ and $x_3$. We claim that if $a_2,a_3\in B_{0}(\ell)$ with $a_2,a_3\ge 0$, then
		\begin{equation}\label{est:A_2-small}
			0\le \iA_2(t,a_1,a_2,a_3) \le \dlt e^{-C_3^{-1} Mt},
		\end{equation}
		and
		\begin{equation}\label{est:A_3-small}
			0\le \iA_3(t,a_1,a_2,a_3) \le \dlt.
		\end{equation}
		Recalling \eqref{equality: iA_2}, \eqref{est: iu}, and $A^{\mathcal{I}}_2(t,a_1,a_2,a_3)\ge0$, we have
		\begin{equation*}
			\frac{d}{dt}A^{\mathcal{I}}_2(t,a_1,a_2,a_3) \le C e^{tM}
			A_2^{\mathcal{I}}(t,a_1,a_2,a_3)\left(1+\log \frac{3L}{A_2^{\mathcal{I}}(t,a_1,a_2,a_3)}\right)-MA^{\mathcal{I}}_2(t,a_1,a_2,a_3),
		\end{equation*}
		Then proceeding as we did  {to derive \eqref{est: A_2 sample}}, we see that for $0\le t \le T_M$,
		\begin{equation*}   
			\log A^{\mathcal{I}}_2(t,a_1,a_2,a_3) 
			\le 1+\log 3L+ C_3^{-1}\left(\log a_2 -(1+\log3L)-Mt \right),
		\end{equation*}
		which is equivalent to
		\begin{equation*}
			A^{\mathcal{I}}_2(t,a_1,a_2,a_3) \le 3Le \left(\frac{a_2}{3Le}\right)^{C_3^{-1}} e^{-C_3^{-1}Mt}.
		\end{equation*}
		But since we assumed $a_2\le \ell \le (3Le)^{1-C_3}\dlt^{C_3}$, \eqref{est:A_2-small} holds. Then, \eqref{est:A_3-small} can be handled by a parallel argument.
		
		Next, we prove our second statement. Suppose that $a\in \mathbb{T}^3\backslash B_{0}(\ell)$. Then $|a_1|\ge \ell$ implies that 
		$\left|\iA_1(t,a_1,a_2,a_3)\right| = |a_1| e^{Mt}\ge \ell$ for all $t$. For $\iA_2(t,a)$ and $\iA_3(t,a)$, we only need to consider the case when $a_2,a_3\ge 0$ by the odd symmetry of $\iw$ in both $x_2$ and $x_3$. We claim that if $a_2,a_3\in \mathbb{T}^3 \backslash B_{0}(\ell)$ with $a_2,a_3\ge 0$, then
		\begin{equation}\label{est:A_2-small-1}
			\iA_2(t,a_1,a_2,a_3) \ge (3Le)^{1-C_3}\ell^{C_3}e^{-C_3Mt},
		\end{equation}
		and
		\begin{equation}\label{est:A_3-small-1}
			\iA_3(t,a_1,a_2,a_3) \ge (3Le)^{1-C_3}\ell^{C_3}.
		\end{equation}
		From \eqref{equality: iA_2}, \eqref{est: iu}, and $A^{\mathcal{I}}_2(t,a_1,a_2,a_3)\ge0$, we have
		\begin{equation*}
			-\frac{d}{dt}A^{\mathcal{I}}_2(t,a_1,a_2,a_3)
			\le C e^{tM}
			A_2^{\mathcal{I}}(t,a_2,a_3)\left(1+\log \frac{3L}{A_2^{\mathcal{I}}(t,a_1,a_2,a_3)}\right)+MA^{\mathcal{I}}_2(t,a_1,a_2,a_3).
		\end{equation*}
		With the same argument as  {the derivation of \eqref{est: A_2 opposite sample}}, we obtain for $0\le t \le T_M$,
		\begin{equation*}
			\log \frac{3L}{A_2^{\mathcal{I}}(t,a_1,a_2,a_3)}\le -1+ C_3\left(1+\log \frac{3L}{a_2}+Mt \right),
		\end{equation*}
		so that
		\begin{equation*}
			A^{\mathcal{I}}_2(t,a_1,a_2,a_3) \ge 3Le\left(\frac{a_2}{3Le}\right)^{C_3} e^{-C_3Mt}\ge
			(3Le)^{1-C_3}\ell^{C_3}e^{-C_3Mt}  {,}
		\end{equation*}
   {where we used the assumption $a_2\ge \ell$ in the last inequality.}
		Similarly, we can show \eqref{est:A_3-small-1}.
		Noticing $T_M=\frac{\log(M+1)}{M}$, our second statement follows.
	\end{proof}

	 {Now we are ready to estimate $\nrmb{\sw(t,\cdot)}_{L^{\infty}}$ near the origin.}
	\begin{lemma}\label{lemma sw-l^infty-linear}
		Let $\ell$ in \eqref{def: sw_0} satisfy $\ell \le  \min\left\{\frac{\dlt}{M+1}, (3Le)^{1-C_3}\dlt^{C_3}\right\}$. Then for $0\le t \le T_M$, we have
		\begin{equation*}
			\nrmb{\sw(t,\cdot)}_{L^{\infty}\left(B_{0}\left(C_4\left(\frac{\ell}{M+1}\right)^{C_3}\right)\right)}\lesssim \varepsilon e^{-Mt}.
		\end{equation*}
	\end{lemma}
	\begin{proof}
		Recalling $\iw\equiv 0$ in $[0,T_M]\times B_{\dlt}(0)$, the  {previous} lemma reduces \eqref{eq: linearized sw with iw} to
		\begin{equation}\label{eq: linearized sw with iw-1}
			\rd_t\sw + (\lu + \iu)\cdot \nb \sw = \nb (\lu + \iu) \sw
		\end{equation}
		in $[0,T_M]\times B_{\dlt}(0)$.
		First of all, we claim that $\sw_1 (t,\cdot) = 0$ in $[0,T_M]\times B_{0}\left(C_4\left(\frac{\ell}{M+1}\right)^{C_3}\right)$.
		Indeed, noticing $\iu_1=0$ (see \eqref{condi: iu}), \eqref{eq: linearized sw with iw-1} gives
		\begin{equation*}
			\rd_t \sw_1 + Mx_1 \rd_1 \sw_1 + (-Mx_2 + \iu_2)\rd_2 \sw_1 + \iu_3\rd_3 \sw_1 = M \sw_1.
		\end{equation*}
		Evaluating along the characteristic $\iA$, Lemma \ref{lemma: iA-small} implies $\sw_1 = 0$ in $[0,T_M]\times B_{0}\left(C_4\left(\frac{\ell}{M+1}\right)^{C_3}\right)$ because $\sw_{1,0} =0$ in $B_0(\ell)$.  Next, \eqref{condi: iu} reduces the equations of $\sw_2$ and $\sw_3$ as follows:
		\begin{equation}\label{eq: reduced sw linearized}
			\left\{ \begin{aligned}
				&\rd_t \sw_2 + Mx_1 \rd_1 \sw_2 + (-Mx_2 + \iu_2)\rd_2 \sw_2 + \iu_3\rd_3 \sw_2 = (-M+\rd_2 \iu_2) \sw_2 + \rd_3\iu_2\sw_3, \\
				&\rd_t \sw_3 + Mx_1 \rd_1 \sw_3 + (-Mx_2 + \iu_2)\rd_2 \sw_3 + \iu_3\rd_3 \sw_3 = \rd_2 \iu_3 \sw_2 + \rd_3\iu_3\sw_3,
			\end{aligned}
			\right.
		\end{equation} so that
			in $[0,T_M]\times B_{0}\left(C_4\left(\frac{\ell}{M+1}\right)^{C_3}\right)$.
			Hence using \eqref{est: nabla iu} and Lemma \ref{lemma: iA-small}, we derive
			\begin{equation}\label{est: spw2}
				\begin{split}
					\rd_t |\sw_2(t,\iA(t,a))| &=\frac{\sw_2(t,\iA(t,a)) \rd_t \left(\sw_2(t,\iA(t,a))\right) }{|\sw_2(t,\iA(t,a))|} \\
					&\le \left(-M+Ce^{-C_3^{-1}Mt}\right)|\sw_2(t,\iA(t,a))| + Ce^{-C_3^{-1}Mt}|\sw_3(t,\iA(t,a))|
				\end{split}
			\end{equation}
			and similarly
			\begin{equation}\label{est: spw3}
				\rd_t |\sw_3(t,\iA(t,a))| \le Ce^{-C_3^{-1}Mt} \left(|\sw_2(t,\iA(t,a))|+|\sw_3(t,\iA(t,a))| \right).
			\end{equation}
			Thus, from
			\begin{equation*}
				\rd_t \left(|\sw_2(t,\iA(t,a))|+ |\sw_3(t,\iA(t,a))|\right) \lesssim e^{-C_3^{-1}Mt} \left(|\sw_2(t,\iA(t,a))|+|\sw_3(t,\iA(t,a))|\right),
			\end{equation*}
			we have
			\begin{equation*}
				|\sw_2(t,\iA(t,a))|+ |\sw_3(t,\iA(t,a))| \lesssim \left|\sw_{2,0}(a) + \sw_{3,0}(a)\right|\lesssim \varepsilon,
			\end{equation*}
			where we used \eqref{def: sw_0} in the last inequality.
			Inserting $|\sw_2(t,\iA(t,a))|\lesssim \varepsilon$ into \eqref{est: spw3}, we obtain
			\begin{equation*}
				\rd_t |\sw_3(t,\iA(t,a))| \le Ce^{-C_3^{-1}Mt} |\sw_3(t,\iA(t,a))|+C\varepsilon e^{-C_3^{-1}Mt}.
			\end{equation*}
			Noticing $\sw_{3,0}(a)=0$ in $B_0(\ell)$, Gr\"onwall's inequality gives
    {
			\begin{equation*}
   \begin{split}
				|\sw_3(t,\iA(t,a))| &\le
    \exp{\left(\int_0^t C e^{-C_3^{-1}Ms}ds\right)}\int_0^t C\varepsilon e^{-C_3^{-1}Ms}ds \\
    &= \exp\left( \frac{C\left(1- e^{-C_3^{-1}Mt}\right)}{C_3^{-1}M}\right)\frac{C\varepsilon\left(1- e^{-C_3^{-1}Mt}\right)}{C_3^{-1}M} \lesssim \frac{\varepsilon}{M}\lesssim \varepsilon e^{-Mt}.
    \end{split}
			\end{equation*}
   }
			for $t\in[0,T_M]$.
			Inserting this into \eqref{est: spw2}, we obtain
			\begin{equation*}
				\rd_t |\sw_2(t,\iA(t,a))| \le \left(-M+Ce^{-C_3^{-1}Mt}\right)|\sw_2(t,\iA(t,a))| + C\varepsilon e^{-(C_3^{-1}+1)Mt},
			\end{equation*}
			which yields
			\begin{equation*}
				\begin{split}
					\rd_t \left(|\sw_2(t,\iA(t,a))|\exp \left(\int_0^t M-Ce^{-C_3^{-1}Ms} ds\right)\right)
					&\le  C\varepsilon e^{-(C_3^{-1}+1)Mt} \exp \left(\int_0^t M-Ce^{-C_3^{-1}Ms} ds\right)\lesssim  \varepsilon e^{-C_3^{-1}Mt}.
				\end{split}
			\end{equation*}
			Integrating from $0$ to $t$, we obtain
			\begin{equation*}
				\begin{split}
					|\sw_2(t,\iA(t,a))| &\le \exp \left(\int_0^t -M+Ce^{-C_3^{-1}Ms} ds\right) \left(\sw_{2,0}(a) +  \int_{0}^t C\varepsilon e^{-C_3^{-1}Ms} ds \right) \lesssim \varepsilon e^{-Mt}. \qedhere 
				\end{split}
			\end{equation*}
	\end{proof}

	\subsection{Comparison between solutions}\label{comparison}
	In this section, we compare our pseudo-solutions with real solutions as we mentioned in the beginning of this section. In order to distinguish solutions  {of} \eqref{eq:euler-l-vorticity-s} and \eqref{eq: linearized sw with iw}, we denote the solution of the linearized equation \eqref{eq: linearized sw with iw} by $(\omega^{\mathcal{S},P,Lin},u^{\mathcal{S},P,Lin})$.
	We fix $s>\frac52$ and set $\bar{u}=\lu+\piu$ in \eqref{eq:euler bu}, $u=\lu+\piu+\psu$ in \eqref{eq:euler just u},  {$\tilde{u}=u-\Bar{u}$ in \eqref{eq:euler tu},} and $\tilde{u}^{Lin}=u^{\mathcal{S},P,Lin}$ in \eqref{eq:euler ltu} to employ Proposition \ref{prop: perturbation}.  {
		\eqref{def: sw_0} implies that $\tilde{u}_0$ in \eqref{eq:euler tu} satisfies
		\begin{equation}\label{tpsi-hsps}
			\normhspsr{\tilde{u}_0} \lesssim \normhsr{\tilde{\psi}}=\ell^{-s+\frac{3}{2}}\normhsr{\psi},
		\end{equation}
	}
	and Lemma \ref{lem: iu-Hs-bound} gives
	\begin{equation}\label{est:baru int}
		\int_0^{T_M} \normhspps{(\lu+\piu)(\tau,\cdot)}ds \lesssim MT_M + \frac{e^{ {C}MT_M}-1}{ {C}M} \lesssim  {M^C},
	\end{equation}
  {by adjusting the value of absolute constant $C>1$ from an inequality to another.} Therefore, Proposition \ref{prop: perturbation} implies
		\begin{equation*}
			\normif{\psw(t,\cdot)-\omega^{\mathcal{S},P,Lin}(t,\cdot)} \lesssim \varepsilon^2 \ell^{-2s+3} e^{CM^C}
		\end{equation*}
		on $[0,T_M]$, whenever $\varepsilon>0$ satisfies
		\begin{equation}\label{eps condi'}
			\varepsilon \le \ell^{s-\frac32} e^{-CM^C}.
		\end{equation}
	Hence, it follows from Lemma \ref{lemma sw-l^infty-linear} that
	\begin{equation}\label{est: sw in small ball}
		\nrmb{\psw(t,\cdot)}_{L^{\infty}\left(B_{0}\left(C_4\left(\frac{\ell}{M+1}\right)^{C_3}\right)\right)}\lesssim \varepsilon e^{-Mt}
	\end{equation}
	 {on $[0,T_M]$
		if we pick $\ell$ and $\varepsilon$ satisfying
		\begin{equation}\label{condi: eps, ell}
			\ell \le  \min\left\{\frac{\dlt}{M+1}, (3Le)^{1-C_3}\dlt^{C_3}\right\},\qquad
			\varepsilon\le \ell^{2s-3} e^{-M^C}
		\end{equation}
		with $C$ adjusted.}  {(Here, we have used that $M \gg 1$.)}
	
	Next, we set $\busu:=\lu+\iu$, so that $\busu$ solves \eqref{eq:euler busu}.  
		Thus, using \eqref{est: nabla iu}, \eqref{tpsi-hsps}, and \eqref{est:baru int}, Proposition \ref{prop: perturbation 2} give us \eqref{eq:nb-u-i-decay} for $\varepsilon,\,\ell$ satisfying \eqref{condi: eps, ell} of which $C$ is adjusted if necessary.  
		To derive \eqref{eq:omega-s-decay}, noticing $u=\lu+\piu+\psu=\lu+\iu+\su$  and recalling \eqref{est:tu}, \eqref{tpsi-hsps}, and \eqref{est:baru int}, there exists a constant $C>0$ such that
		\begin{equation*}
			\frac{d}{dt}\left|\Phi(t,a)-\iA(t,a) \right| \le \normif{\psu}
    {\lesssim \normhsps{\psu}}
   \lesssim  \varepsilon\ell^{-s+\frac32} e^{CM^C}
		\end{equation*}
		for $\varepsilon>0$ satisfying \eqref{eps condi'},
		where $\Phi, \iA$ are from \eqref{def: flow Phi}, \eqref{def: iA-extended}, respectively.
		Thus if $\varepsilon,\, \ell$ satisfy \eqref{condi: eps, ell} of which $C$ is adjusted if necessary, 
   {then the estimate
  \begin{equation*}
      \left|\Phi(t,a)-\iA(t,a) \right| \lesssim \varepsilon\ell^{-s+\frac32} e^{CM^C}t \lesssim \sqrt{\varepsilon}
  \end{equation*} on $[0,T_M]$ and
  Lemma \ref{lemma: iA-small} imply }
		\begin{equation*}
			\left|\Phi(t,a)\right|\le \left|\iA(t,a) \right|+\left|\Phi(t,a)-\iA(t,a) \right|\le \dlt + \sqrt{\varepsilon} \le 2\dlt
		\end{equation*}
		for $(t,a)\in [0,T_M] \times B_0(\ell)$ while
		\begin{equation*}
			\left|\Phi(t,a)\right|\ge \left|\iA(t,a) \right|-\left|\Phi(t,a)-\iA(t,a) \right|\ge C_4\left( \frac{\ell}{M+1}\right)^{C_3}- \sqrt{\varepsilon} \ge C_5 \left( \frac{\ell}{M+1}\right)^{C_3}
		\end{equation*}
		for $(t,a)\in [0,T_M] \times \bbT^3\backslash B_0 (\ell)$ and some $C_5>0$.
		Recall that $\piw =0$ in $B_0(2\dlt)\times [0,T_M]$ (see Lemma \ref{lemma: iA} and Remark \ref{rmk: dlt}). Hence 
		by \eqref{eq:euler-l-vorticity-s}, $\psw$ solves
		\begin{equation*}
			\rd_{t} \left(\psw(t,\Phi(t,a))\right) = \nb u (t,\Phi(t,a)) \psw (t,\Phi(t,a))
		\end{equation*}
		for  $(t,a)\in [0,T_M] \times B_0(\ell)$. Since $\sw$ also solves
		\begin{equation*}
			\rd_{t} \left(\sw(t,\Phi(t,a))\right) = \nb u (t,\Phi(t,a)) \sw (t,\Phi(t,a)),
		\end{equation*}
		we have
		\begin{equation*}
			\rd_{t} \left|\psw(t,\Phi(t,a))-\sw(t,\Phi(t,a)) \right| \le \normif{\nb u} \left|\psw(t,\Phi(t,a))-\sw(t,\Phi(t,a)) \right|
		\end{equation*}
		for $(t,a)\in[0,T_M] \times B_0(\ell)$. But $\normif{\nb u} < \infty$ up to time $T_M$, and $\psw$ and $\sw$ have the same initial data $\sw_0$, so that
		$\psw(t,\Phi(t,a))=\sw(t,\Phi(t,a))$ for $(t,a)\in[0,T_M] \times B_0(\ell)$. This implies
		$\psw(t,x)=\sw(t,x)$ for $(t,x)\in[0,T_M] \times B_0\left(C_5 \left( \frac{\ell}{M+1}\right)^{C_3} \right)$, and therefore \eqref{eq:omega-s-decay} follows from \eqref{est: sw in small ball}.  This completes our proof of Theorem \ref{thm: presence}. $\Box$

	\medskip

	\noindent \textbf{Acknowledgments}. Research of TY  was partially supported by Grant-in-Aid for Scientific Research B (20H01819), Japan Society for the Promotion of Science (JSPS). IJ has been supported by the National Research Foundation of Korea(NRF) grant No. 2022R1C1C1011051. 
	
	\bibliographystyle{amsplain}
	
	
\end{document}